\documentclass[12pt]{amsart}
\usepackage[T1]{fontenc}
\usepackage[utf8]{inputenc}
\usepackage{tabularx}
\usepackage{amsmath}
\usepackage{amsthm,amssymb,color,comment}
\usepackage{bbm}
\usepackage{xfrac}
\usepackage{hyperref}
\usepackage{dsfont}
\usepackage{tikz}
\usepackage{enumitem}
\usepackage[normalem]{ulem}
\usepackage[sort]{cite}
\usepackage{todonotes}
\usepackage{graphics}
\usepackage{lmodern}
\usepackage[english]{babel}
\usepackage{amsfonts}
\usepackage{color}
\usepackage{floatrow}
\usepackage[ruled,vlined]{algorithm2e}
\usepackage{subcaption}
\usepackage{geometry}
\usepackage{enumitem}
\usepackage{stackengine}
\usepackage{bigints}
\usepackage{mathtools}
\stackMath

\newcommand{\R}{\mathbb{R}}

\newcommand{\T}{\mathbb{T}}

\newtheorem{theorem}{Theorem}[section]

\newtheorem{remark}[theorem]{Remark}

\newtheorem{proposition}[theorem]{Proposition}

\newtheorem{lemma}[theorem]{Lemma}

\DeclareMathAlphabet{\mathup}{OT1}{\familydefault}{m}{n}
\newcommand{\dx}[1]{\mathop{}\!\mathup{d} #1}
\newcommand{\ddt}{\frac{\dx{}}{\dx{t}}}
\renewcommand{\div}{\operatorname{div}}
\DeclareMathOperator{\supp}{supp}
\DeclareMathOperator{\sign}{sign}
\DeclarePairedDelimiter{\prt}{(}{)}
\DeclarePairedDelimiter{\abs}{\lvert}{\rvert}
\DeclarePairedDelimiter{\norm}{\lVert}{\rVert}
\DeclarePairedDelimiter{\set}{\{}{\}}
\DeclarePairedDelimiter{\brk}{[}{]}
 
\newcommand{\2}{_{\gamma_{2}}}
\newcommand{\1}{_{\gamma_{1}}}
\newcommand{\g}{_{\gamma}}
\newcommand{\ga}{_{\gamma,\delta}}
\newcommand{\gmin}{\gamma_{\mathrm{min}}}
\newcommand{\gmax}{\gamma_{\mathrm{max}}}

\author{Tomasz D\k{e}biec}
\address{{\it Tomasz D\k{e}biec:} Faculty of Mathematics, Informatics and Mechanics, University of Warsaw, ul. Banacha 2, 02-097 Warsaw, Poland}
\email{t.debiec@mimuw.edu.pl}
\author{Piotr Gwiazda}
\address{{\it Piotr Gwiazda:} Institute of Mathematics, Polish Academy of Sciences, ul. \'Sniadeckich 8, 00-656 Warsaw, Poland}
\email{pgwiazda@mimuw.edu.pl}
\author{B\l{}a\.zej Miasojedow}
\address{{\it B\l{}a\.zej Miasojedow:} Faculty of Mathematics, Informatics and Mechanics, University of Warsaw, ul. Banacha 2, 02-097 Warsaw, Poland}
\email{b.miasojedow@mimuw.edu.pl}
\author{Zuzanna~Szyma\'nska}
\address{{\it Zuzanna Szyma\'nska:} Interdisciplinary Centre for Mathematical and Computational Modelling, University of Warsaw, ul. Tyniecka 15/17, 02-630 Warsaw, Poland}
\email{z.szymanska@icm.edu.pl}

\begin{document}
\title[Stability of solutions of porous medium equations]{Stability of solutions of the porous medium equation with growth with respect to the diffusion exponent}

\begin{abstract}
We consider a macroscopic model for the growth of living tissues incorporating pressure-driven dispersal and pressure-modulated proliferation. Assuming a power-law relation between the mechanical pressure and the cell density, the model can be expressed as the porous medium equation with a growth term. We prove H\"older continuous dependence of the solutions of the model on the diffusion exponent. The main difficulty lies in the degeneracy of the porous medium equations at vacuum. To deal with this issue, we first regularise the equation by shifting the initial data away from zero and then optimise the stability estimate derived in the regular setting.
\end{abstract}

\maketitle

\section{Introduction}
The use of mechanical fluid-based models for tissue growth has become a well-established approach in modelling tumour growth~\cite{PQV14,ByrneDrasdo,Friedman2007,Goriely}. 
In the simplest setting, the tumour is limited only by the availability of space and the dynamics of cell densities are driven by a self-reinforcing interplay between proliferation (which generates mechanical pressure) and dispersal. In this realm, we investigate the following equation for the evolution of the cell number density $u_\gamma$:
\begin{equation}
    \label{eq:PDEintro}
    \partial_t u_\gamma + \div\prt{u_\gamma v_\gamma} = u_\gamma G(p_\gamma),
\end{equation}
where $v_\gamma$ is the velocity, related to the pressure, $p_\gamma$, by Darcy's law $v_\gamma = -\nabla p_\gamma$. Closing the equation with a constitutive relation $p_\gamma = p_\gamma(u_\gamma)$, we end up with a general filtration equation for the movement of cells (the porous medium being the surrounding extracellular matrix). The function $G$ is a growth term incorporating the birth and death of cells. The subscript $\gamma$ reflects that in general the model depends on a vector of parameters; however, for our current purposes, it is a single parameter expressing the ''stiffness'' of the pressure in the state equation~\eqref{PressureLaw} below.

One of the main challenges associated with the use of mathematical models in practice is the lack of model validations against data and the enormous difficulty in their proper calibration. This applies, in particular, to tumour growth models which, despite these shortcomings, gave a clear step forward in our understanding of the disease dynamics. The next step requires developing mathematical methods enabling proper validation and calibration to create in the end a fully predictive model. Following this concept, we focus on probably the most important and well-established class of models describing cancer development which are models of cancer invasion. At a certain point in the development of a solid tumour, cells from the primary cancerous mass start to migrate and invade the tissue surrounding the tumour. This initial invasion of the local tissue is the first step in the complex process of secondary spread wherein the cancer cells migrate to other locations in the body and set up new tumours called metastases. These secondary tumours are responsible for above 90\% of all (human) deaths from cancer~\cite{Fares2020}. Although molecular mechanisms and cell-scale migration dynamics are roughly known, the invasion process as a whole is very complex and despite many efforts not fully understood \cite{Lowengrub2010}.

Even before finding the correct parameter values, the big challenge lies in the mentioned validation, that is, in obtaining adequate descriptions for the individual sub-processes that contribute to the whole biophysical process. For instance, within the invasion process, one can distinguish its basic components: the adhesion of tumour cells to the extracellular matrix (the extracellular matrix is the substance that fills the space between cells and binds cells into tissues and organs), the secretion of matrix-degrading enzymes by tumour cells, the migration of tumour cells, and finally their proliferation. Considering the complexity of the problem it seems reasonable to split it into smaller, still challenging, however feasible to achieve, sub-tasks. And so, our objective is to adequately describe the process of cancer cell proliferation, which is a critical step in the realistic modelling of cancer invasion. In this paper, we derive a stability estimate, which can later serve as a vital building block in deterministic or stochastic parameter estimation.

Let us point out that similar stability studies were previously undertaken in the framework of entropy solutions for a variety of hyperbolic conservation laws and quasilinear parabolic equations, notably with applications to traffic flow modelling and gas dynamics~\cite{BianchiniColombo, ChenCleo, ChenKarlsen}. While it is known that the solutions of the filtration equation $\partial_t u = \Delta\phi(u)$ depend continuously on the nonlinearity~\cite{BenilanCrandall}, to our knowledge, no explicit modulus of continuity is known. 

The paper is organised as follows. In the next section, we give the precise formulation of our model, state all necessary assumptions and formulate the main result. In Section~\ref{sec:A_priori}, we discuss the key properties of the solutions of the model. In Section~\ref{sec:Regularisation} we regularise the equation by shifting the initial data away from vacuum. We prove that solutions to such regularised problem stay strictly positive and provide necessary regularity estimates. In Section~\ref{sec:Stability_regular} we compute the $L^1$-stability estimate for the regularised solutions. Then, in Section~\ref{sec:MainProof} we use that estimate to conclude the proof of the main theorem. Finally, in Section~\ref{sec:Discussion}, we discuss some perspectives for further research.

\section{Formulation of the problem and main result}
\label{sec:Formulation}

\noindent We consider the initial value problem 
\begin{equation}
  \label{nonlinear_diffusion_0}
  \left\{
\begin{aligned}
\partial_t u_{\gamma}(t,x) &=  \div\prt*{u_{\gamma}(t,x) \nabla p_\gamma(t,x)} + u_{\gamma}(t,x) G(p_\gamma(t,x)),\quad x \in \R^d, \;\; t\in[0,T],\\
u_\gamma(0,x) &= u_{0,\gamma}(x) \geq 0,
\end{aligned}
    \right.
\end{equation}
where $u_\gamma$ represents the density of cells and $p_\gamma$ stands for the pressure within the tissue, obeying the following constitutive relation
\begin{equation}
\label{PressureLaw}
    p_\gamma = \frac{\gamma}{\gamma-1}u_\gamma^{\gamma-1}.
\end{equation}
Using this relation, equation~\eqref{nonlinear_diffusion_0} can be equivalently written as a nonlinear diffusion equation
\begin{equation}
  \label{nonlinear_diffusion}
  \left\{
\begin{aligned}
\partial_t u_{\gamma}(t,x) &=  \Delta\prt*{u_{\gamma}(t,x)}^{\gamma} + u_{\gamma}(t,x) G(p_\gamma(t,x)),\quad x \in \R^d, \;\; t\in[0,T],\\
u_\gamma(0,x) &= u_{0,\gamma}(x) \geq 0.
\end{aligned}
    \right.
\end{equation}

\noindent We assume that the growth function $G$ has the following properties:
\begin{equation}
\label{GrowthFunction}
    G\in C^1([0,\infty)),\quad G(0)>0,\quad G'<0,\quad G(p_H) = 0,
\end{equation}
for some \textit{homeostatic pressure} $p_H>0$, and we make the following assumptions on the initial data $u_{0,\gamma}$:\ there exists a compact set $K\subset\R^d$ such that
\begin{equation}
  \left\{  
\begin{aligned}
\label{InitialData}
    &u_{0,\gamma} \geq 0,\quad u_{0,\gamma} \in L^1(\R^d)\cap L^\infty(\R^d),\quad \supp u_{0,\gamma}\subset K\\
    &\Delta (u_{0,\gamma})^\gamma \in L^1(\R^d),\quad p_\gamma(u_{0,\gamma}) \leq p_H,
\end{aligned}
\right.
\end{equation}
uniformly in $\gamma>1$.
Under these assumptions, there exists a unique bounded weak solution $u\g, p\g\in L^\infty(0,T;L^1(\R^d)\cap L^\infty(\R^d))$.
Existence was shown in~\cite[Theorem~9.3]{Vazquez2007} (without the proliferation term; however, under assumption~\eqref{GrowthFunction}, it is not difficult to extend the result to equation~\eqref{nonlinear_diffusion}) and in~\cite{GPS2019}; uniqueness follows from Proposition~\ref{prop:ContractionPrinciple}.

Model~\eqref{nonlinear_diffusion_0} has attracted a lot of attention in recent years. In~\cite{PQV14, KP18} it was shown that in the incompressible limit $\gamma\to\infty$ one obtains a Hele-Shaw-type free boundary problem and the rate for this convergence was computed in~\cite{DDP}. It can also be shown that~\eqref{nonlinear_diffusion_0} arises as the inviscid limit of a viscoelastic tumour growth model~\cite{PerthameVauchelet, DDMS,elbar2023}. Systems incorporating multi-species dynamics are also of great the\-o\-re\-ti\-cal and practical interest, although they pose substantial additional mathematical difficulties, for instance, due to the formation of sharp interfaces between the species, see~\cite{GPS2019, PriceXu, DDMS, DebiecSchmidtchen2020, DebiecEtAl, Jacobs, CFSS17} for some recent developments in the field.

The main contribution of this paper is the formulation and proof of the following theorem on H\"older continuous dependence of the solution on the exponent $\gamma$.
\begin{theorem}\label{thm:main}
Let $1< \gmin < \gmax < \infty$ be given and set $\Gamma=[\gmin,\gmax]$. For $\gamma_1, \gamma_2 \in \Gamma$,
let $u_{\gamma_1}$ and $u_{\gamma_2}$ be two solutions to~\eqref{nonlinear_diffusion} with initial data $u_{0,\gamma_1}$ and $u_{0,\gamma_2}$ satisfying~\eqref{InitialData}. 
There exist constants $C_1$ that depends on the final time $T$ and $C_2$ that depends on $\Gamma$, the initial data, and $T$, such that
\begin{equation*}
\label{ModContinuity}
\sup_{t\in[0,T]\;}\| u_{\gamma_1}(t) - u_{\gamma_2}(t)\|_{L^1(\R^d)} \leq C_1\norm*{u_{0,\gamma_1}-u_{0,\gamma_2}}_{L^1(\R^d)} + C_2|\gamma_2-\gamma_1|^{\frac{1}{\min(\gamma_1,\gamma_2)+1}}.
\end{equation*}
\end{theorem}

\begin{remark}
    Although it does not satisfy conditions~\eqref{GrowthFunction}, the function $G\equiv 0$ is a viable choice for our analysis. Thus, Theorem~\ref{thm:main} holds for the classical porous medium equation.
\end{remark}

\section{Basic properties of the solution}
\label{sec:A_priori}
\noindent Let us first introduce some auxiliary notation. By $|u|_{-}$ and $\sign_{-}(u)$ we denote the negative part and negative sign of the function $u$, i.e.,
    \begin{equation*}
        |u|_{-} :=
        \begin{cases}
            -u, & \text{for $u<0$,}\\
            0, & \text{for $u\geq 0$,}
        \end{cases}
        \qquad \text{and}\qquad \sign_{-}(u):=
        \begin{cases}
            -1, & \text{for $u<0$,}\\
            0, & \text{for $u\geq 0$.}
        \end{cases}
    \end{equation*}
Analogously, we define the positive part and the positive sign of $u$
    \begin{equation*}
        |u|_{+} :=
        \begin{cases}
            u, & \text{for $u>0$,}\\
            0, & \text{for $u\leq 0$,}
        \end{cases}
        \qquad \text{and}\qquad \sign_{+}(u):=
        \begin{cases}
            1, & \text{for $u>0$,}\\
            0, & \text{for $u\leq 0$.}
        \end{cases}
    \end{equation*}
The following fundamental properties of equation~\eqref{nonlinear_diffusion_0} are well known, but we recall them together with proofs for the reader's convenience. The arguments we give below are formal in that we do not fully justify some computations. They can be made rigorous using well-known renormalisation techniques.
\begin{proposition}[Nonnegativity of the density]
    Let $u_\gamma$ be the solution to equation~\eqref{nonlinear_diffusion}. Then, $u_\gamma(t)\geq 0$ for all $t\geq 0$.
\end{proposition}
\begin{proof}
    Multiplying equation~\eqref{nonlinear_diffusion_0} by $\sign_{-}(u\g)$, we have
    \begin{equation*}
        \partial_t |u\g|_{-} - \div\prt*{|u\g|_{-}\nabla p\g} \leq |u\g|_{-} G(p\g),
    \end{equation*}
    which, upon integrating in space, becomes
    \begin{equation*}
        \ddt\int_{\R^d} |u\g|_{-} \dx x \leq C\int_{\R^d} |u\g|_{-} \dx x.
    \end{equation*}
    Gronwall's lemma yields $|u\g|_{-}\equiv 0$.
\end{proof}
\begin{proposition}[Compact support property]
\label{prop:CompactSupport}
    There exists a bounded open set $\Omega\subset\R^d$, depending on $\sup_{\gamma}\norm*{u_{0,\gamma}}_{L^\infty(\R^d)}$, the size of the initial support $K$, and the final time $T$, such that
    \begin{equation*}
        \supp u_\gamma(t) \subset \Omega\quad\forall t\in[0,T].
    \end{equation*}
\end{proposition}
\begin{proof}
The following proof is taken from~\cite{PQV14}.
The compact support property for the densities is deduced from that of the pressures. Using~\eqref{PressureLaw} and~\eqref{nonlinear_diffusion}, we see that the pressure $p\g$ satisfies the equation
    \begin{equation*}
        \partial_t p\g = \gamma p\g\Delta p\g + \abs*{\nabla p\g}^2 + \gamma p\g G(p\g).
    \end{equation*}
Since $G$ is decreasing, it follows that $p\g$ is a subsolution of the equation
    \begin{equation}
    \label{eq:PressureSub}
        \partial_t p\g = \gamma p\g\Delta p\g + \abs*{\nabla p\g}^2 + \gamma p\g G(0).
    \end{equation}
Let us now set $\tau:=d/(4G(0))$ and consider the function  
\begin{equation*}
    P(t,x) := \abs*{C - \frac{|x|^2}{4(\tau+t)}}_{+},
\end{equation*}
where the constant $C$ is chosen large enough so that $P(0,x) \geq p\g(u_{0,\gamma}(x))$ for all $x\in\R^d$ and all $\gamma>1$. Such a choice is possible since the initial densities are assumed to be uniformly supported in the compact set $K$. 
A direct calculation shows that
\begin{equation*}
    \partial_t P - \gamma P\Delta P - \abs*{\nabla P}^2 - \gamma P G(0) = \gamma P\prt*{\frac{d}{2(\tau+t)}-G(0)}\geq 0,
\end{equation*}
so that $P$ is a supersolution of~\eqref{eq:PressureSub} on the time interval $\brk*{0,\frac{d}{4G(0)}}$. It follows that $\supp p\g\subset \supp P$, and consequently the support of $u\g$ is uniformly contained in the ball of radius $\sqrt{2Cd/G(0)}$ for all $t\in \brk*{0,\frac{d}{4G(0)}}$. Since $\norm*{p\g(\frac{d}{4G(0)})}_{L^\infty(\R^d)}\leq c$, the same argument can be reiterated until the final time $T$ is reached in a finite number of steps.
\end{proof}
Finally, we observe the following $L^1$-contraction property, which allows to compare the distance between $u_\gamma$ and the regularised density $u_{\gamma,\delta}$ introduced in the next section.
\begin{proposition}[$L^1$-contraction estimate]
\label{prop:ContractionPrinciple}
    For two solutions $u_\gamma^1$ and $u_\gamma^2$ of~\eqref{nonlinear_diffusion} with initial data $u_{0,\gamma}^1$ and $u_{0,\gamma}^2$, respectively, we have
    \begin{equation*}
        \norm*{u_\gamma^1(t) - u_\gamma^2(t)}_{L^1(\R^d)} \leq e^{G_{\mathrm{max}}t}\norm*{u_{0,\gamma}^1 - u_{0,\gamma}^2}_{L^1(\R^d)}\quad \forall t\in[0,T],
    \end{equation*}
    where $G_{\mathrm{max}}$ is the supremum of $|G(p)|$ over $0\leq p\leq \sup_{\gamma}\norm{p\g}_{L^\infty(\R^d)}$.
\end{proposition}
\begin{proof}
    First, we write the equation for the difference of the two solutions:
    \begin{equation*}
        \partial_t\prt*{u\g^1-u\g^2} - \Delta\prt*{\prt*{u\g^1}^{\gamma}-\prt*{u\g^2}^{\gamma}} = \prt*{u\g^1-u\g^2}G(p\g^1) + u\g^2\prt*{G(p\g^1)-G(p\g^2)}.
    \end{equation*}
    Multiplying by $\sign(u\g^1-u\g^2) = \sign\prt*{\prt*{u\g^1}^{\gamma}-\prt*{u\g^2}^{\gamma}} = \sign(p\g^1-p\g^2)$, we obtain
    \begin{equation*}
        \partial_t\abs*{u\g^1-u\g^2} - \Delta\abs*{\prt*{u\g^1}^{\gamma}-\prt*{u\g^2}^{\gamma}} \leq \abs*{u\g^1-u\g^2}G(p\g^1) + u\g^2\prt*{G(p\g^1)-G(p\g^2)}\sign(p\g^1-p\g^2).
    \end{equation*}
    Notice that the last term is non-positive since $G$ is a decreasing function. Integrating in space gives
    \begin{equation*}
        \ddt\int_{\R^d} \abs*{u\g^1-u\g^2} \dx x \leq G_{\mathrm{max}}\int_{\R^d} \abs*{u\g^1-u\g^2} \dx x,
    \end{equation*}
    and the result follows by Gronwall's lemma.
\end{proof}

\section{Regularisation}\label{sec:Regularisation}
\noindent Since the densities we consider are not bounded away from the vacuum, equation~\eqref{nonlinear_diffusion} degenerates near $u_\gamma \approx 0$. This degeneracy causes substantial difficulties in establishing stability with respect to the diffusion exponent. For this reason, in the proof of the main result, we will consider solutions emanating from modified initial data. To preserve integrability, we first restrict the problem to the periodic case. By virtue of Proposition~\ref{prop:CompactSupport}, we can find $L = L(u_{0,\gamma},T)>0$ such that $\Omega\subset[-L,L]^d$. 
Then, in particular, $\supp u_{0,\gamma} \subset [-L,L]^d$ uniformly in $\gamma$. Therefore, we can extend the initial condition periodically outside of the $[-L,L]^d$ cube. Imposing periodic boundary conditions, we can then consider equation~\eqref{nonlinear_diffusion} posed on the flat torus $\T^d:=[-L,L]^d$. We then clearly have
\begin{equation*}
    \norm*{u\g(t)}_{L^1(\R^d)} = \norm*{u\g(t)}_{L^1(\T^d)}\quad\forall t\in[0,T].
\end{equation*}
Now, we shift the initial data to ensure strict positivity
\begin{equation*}
    \label{RegularisedData}
    u_{0,\gamma,\delta} := u_{0,\gamma} + \delta,\quad \delta >0,
\end{equation*}
and consider the regularised problem
\begin{equation}
\label{nonlinear_diffusion_regular}
    \left\{
    \begin{aligned}
        \partial_t u\ga - \Delta u\ga^{\gamma} &= u\ga G(p\ga),\\
        u\ga(0) &= u_{0,\gamma,\delta},
    \end{aligned}
    \right.
\end{equation}
posed in $[0,T]\times\T^d$\footnote{Alternatively, instead of considering a periodic extension, one could pose the regularised problem inside a ball containing the set $\Omega$ and impose homogeneous Neumann boundary condition.}. Here, $p\ga:=p_\gamma(u\ga)$.
Since $u_{0,\gamma,\delta} \geq \delta > 0$ is bounded away from $0$, \eqref{nonlinear_diffusion_regular} becomes strictly parabolic and admits a unique strictly positive strong solution.
The parameter $\delta$ will in the end be chosen in terms of the two exponents $\gamma_1$ and $\gamma_2$, but it will always be bounded by a positive constant depending only on the size of the set $\Gamma$. Therefore, there exists a positive constant $P_H>p_H$ such that
\begin{equation*}
    p_{0,\gamma,\delta} = p_\gamma(u_{0,\gamma,\delta}) \leq P_H.
\end{equation*}
Let us point out that, as $\delta\to0$, $u\ga$ converges to the solution $u_\gamma$ of~\eqref{nonlinear_diffusion} with initial data $u_{0,\gamma}$ in the weak-star topology on $L^\infty((0,T)\times\R^d)$, see~\cite[Theorem~~2]{GPS2019}.

We now prove that indeed the lower bound is preserved.
\begin{proposition}[$L^\infty$ bounds and strict positivity of the regularised density]
    Let $u\ga$ be the solution of~\eqref{nonlinear_diffusion_regular}. Then
\begin{equation*}
    u\ga \leq u_H:=\prt*{\frac{\gamma-1}{\gamma}P_H}^{\frac{1}{\gamma-1}},\quad p\ga \leq P_H,
\end{equation*}
and
\begin{equation*}
    u\ga(t) \geq \delta e^{-G_Mt} > 0\quad \forall t\in (0,T],
\end{equation*}
where $G_M:=\max_{0\leq p\leq P_H}|G(p)|$.
\end{proposition}
\begin{proof}
Since $u_H$ is a constant, we can write
\begin{align*}
    \partial_t\prt*{u\ga - u_H} &- \Delta\prt*{u\ga^\gamma - u_H^\gamma} \\&= \prt*{u\ga-u_H}G(p\ga) + u_H G(p\ga)\\
    &= \prt*{u\ga-u_H}G(p\ga) + u_H\prt*{G(p\ga)-G(P_H)} + u_HG(P_H).
\end{align*}
Multiplying this equation by $\sign_{+}(u\ga-u_H)$, we obtain 
\begin{align*}
    \partial_t\abs*{u\ga - u_H}_{+} - \Delta\abs*{u\ga^\gamma - u_H^\gamma}_{+} &\leq \abs*{u\ga-u_H}_{+}G(p\ga)\\
    &\;+ u_H\prt*{G(p\ga)-G(P_H)}\sign_{+}(u\ga-u_H)\\
    &\;+ u_HG(P_H)\sign_{+}(u\ga-u_H).
\end{align*}
The last two terms are nonpositive, since $\sign_{+}(u\ga-u_H) = \sign_{+}(p\ga-P_H)$ and $G$ is decreasing and since $G(P_H)\leq 0$, respectively. Integrating in space, we thus obtain
\begin{equation*}
    \ddt \int_{\T^d}\abs*{u\ga - u_H}_{+}\dx x \leq C\int_{\T^d}\abs*{u\ga - u_H}_{+}\dx x.
\end{equation*}
Since, by assumption, $u\ga(0)\leq u_H$, it follows that $u\ga \leq u_H$ and $p\ga \leq p_H$ almost everywhere in $(0,T)\times\T^d$.

To prove the lower bound, we define $\underline{u}(t):=\delta e^{-G_Mt}$. It then satsfies $\underline{u}' = -G_M \underline{u}$ and we can write
\begin{align*}
    \partial_t\prt*{\underline{u} - u\ga} - \Delta\prt*{\underline{u}^\gamma - u\ga^\gamma} &= -G_M\underline{u} - u\ga G(p\ga)\\
    &=G(p\ga)\prt*{\underline{u}-u\ga} + \underline{u}\prt*{-G_M-G(p\ga)}.
\end{align*}
Notice that the last term is negative. Upon multiplication by $\sign_{+}\prt*{\underline{u}-u\ga}$ and integration in space, we obtain
\begin{equation*}
    \ddt\int_{\T^d}\abs*{\underline{u}-u\ga}_{+}\dx x \leq G_M\int_{\T^d} \abs*{\underline{u}-u\ga}_{+} \dx x. 
\end{equation*}
At $t=0$, we have $\underline{u}(0) = \delta\leq u_{0,\gamma,\delta}$. Therefore, by Gronwall's lemma, we deduce $\underline{u}(t)\leq u\ga$ for all $t\in[0,T]$.
\end{proof}

In the interest of readability, we henceforth omit the subscript $\delta$. It will be reinstated later, in the conclusion of the proof of Theorem~\ref{thm:main}.
\begin{lemma}
    The following regularity estimates hold uniformly in $\gamma$ and $\delta$:
    \begin{align*}
        \partial_t u_\gamma &\in L^\infty(0,T;L^1(\T^d)),\\
        \Delta u_\gamma^\gamma &\in L^\infty(0,T;L^1(\T^d)),\\
        \nabla u_\gamma^\gamma &\in L^2((0,T)\times\T^d).
    \end{align*}
\end{lemma}
\begin{proof}
    We first establish the bound on the time derivative, the other two then follow easily. To this end, we differentiate equation~\eqref{nonlinear_diffusion_regular} to obtain
    \begin{align*}
        \partial_t\prt*{\partial_t u\g} - \gamma\Delta\prt*{u\g^{\gamma-1}\partial_t u\g} = \partial_t u\g G(p\g) + u\g G'(p\g) \partial_t p\g.
    \end{align*}
    Multiplying by $\sign(\partial_t u\g)$ and observing that $\sign(\partial_t u\g)=\sign(\partial_t p\g)$, we have
    \begin{align*}
        \partial_t\abs*{\partial_t u\g} -  \gamma\Delta\prt*{u\g^{\gamma-1}\abs*{\partial_t u\g}} \leq \abs*{\partial_t u\g} G(p\g) + u\g G'(p\g) \abs*{\partial_t p\g},
    \end{align*}
    where the last term is negative. Integrating in space yields
    \begin{equation*}
        \ddt \int_{\T^d}\abs*{\partial_t u\g}\dx x \leq C\int_{\T^d}\abs*{\partial_t u\g}\dx x,
    \end{equation*}
    and we conclude that $\partial_t u_\gamma \in L^\infty(0,T;L^1(\R^d))$ by Gronwall's lemma. The corresponding bound on the Laplacian follows from equation~\eqref{nonlinear_diffusion_regular}. Finally, integrating by parts, we have
    \begin{equation*}
        \int_0^T\!\!\int_{\T^d} \abs*{\nabla u\g^\gamma}^2\dx x\dx t = -\int_0^T\!\!\int_{\T^d} u\g^\gamma\Delta u\g^\gamma\dx x\dx t,
    \end{equation*}
    and the $L^2$ bound for the gradient follows.
\end{proof}

\section{Stability estimate for regularised densities.}
\label{sec:Stability_regular}
From now on, we suppose that two exponents $\gamma_1, \gamma_2\in \Gamma$ are chosen. Moreover, since $s\geq 1 \Rightarrow s^a\geq 1,\;\forall a>0$, we assume that $\abs*{\gamma_2-\gamma_1}<1$.
We begin with a preliminary calculation.
\begin{lemma}
\label{lem:LapEstimate}
    Let $u\1$ be the solution of~\eqref{nonlinear_diffusion_regular} with initial data $u_{0,\gamma_1}$. Then,
    \begin{equation*}
        \label{LapEstimate}
        \norm*{\Delta u\1^{\gamma_1}-\Delta u\1^{\gamma_2}}_{L^1((0,T)\times\T^d)} \leq C(\Gamma, T)\prt*{1+|\ln\delta|+\delta^{-a}}\abs*{\gamma_2-\gamma_1}
    \end{equation*}
    with 
    \begin{equation*}
    a := 2\min\prt*{\gamma_1,\gamma_2}-\max\prt*{\gamma_1,\gamma_2}.
    \end{equation*}
\end{lemma}
\begin{proof}
Without loss of generality, we assume that $\gamma_2\geq\gamma_1$. Let us observe the following pointwise identities.
\begin{equation}
\label{eq:L_estimate}    
\begin{aligned}
    \Delta u\1^{\gamma_1} - \Delta u\1^{\gamma_2} &= \Delta u\1^{\gamma_1} - \Delta\prt*{ u\1^{\gamma_1}}^{\frac{\gamma_2}{\gamma_1}}\\
    &= \Delta u\1^{\gamma_1}\prt*{1 - \frac{\gamma_2}{\gamma_1}u\1^{\gamma_2-\gamma_1}} - \frac{\gamma_2}{\gamma_1}\frac{\gamma_2-\gamma_1}{\gamma_1}u\1^{\gamma_1\prt*{\frac{\gamma_2}{\gamma_1}-2}}\abs*{\nabla u\1^{\gamma_1}}^2\\
    &\equiv A\Delta u\1^{\gamma_1} - B\abs*{\nabla u\1^{\gamma_1}}^2.
\end{aligned}
\end{equation}
It remains to estimate the $L^\infty$ norm of factors $A$ and $B$. To this end, let us recall that the regularised densities are bounded from below, $u\g\geq \delta e^{-G_MT}>0$.
Since the exponential function is locally Lipschitz, we have
\begin{align*}
    |A| = \abs*{1 - \frac{\gamma_2}{\gamma_1}u\1^{\gamma_2-\gamma_1}} = \abs*{e^0 - e^{\ln\prt*{\frac{\gamma_2}{\gamma_1}u\1^{\gamma_2-\gamma_1}}}} &\leq e^{\abs*{\ln\prt*{\frac{\gamma_2}{\gamma_1}\norm*{u\1(t)}_{L^\infty(\T^d)}^{\gamma_2-\gamma_1}}}}\abs*{\ln\prt*{\frac{\gamma_2}{\gamma_1}u\1^{\gamma_2-\gamma_1}}}\\
    &\leq C(\Gamma,T)\prt*{\ln\prt*{\frac{\gamma_2}{\gamma_1}} + \prt*{\gamma_2-\gamma_1}\abs*{\ln u\1}}.
\end{align*}
Using the elementary bound $\ln s \leq s-1$ for $s\geq 1$, we estimate
\begin{equation*}
    \ln\prt*{\frac{\gamma_2}{\gamma_1}} \leq \frac{\gamma_2-\gamma_1}{\gamma_1} \leq \frac{1}{\gmin}\prt*{\gamma_2-\gamma_1},
\end{equation*}
while using the lower and upper bounds on $u\1$,
\begin{equation*}
    \abs*{\ln u\1(t)} \leq \abs*{\ln\norm{u\1(t)}_{L^\infty(\T^d)}}+ G_MT + \abs*{\ln\delta} \leq C(T) + \abs*{\ln\delta}.
\end{equation*}
Therefore, we have
\begin{equation}
\label{eq:A_estimate}
    |A| \leq C(\Gamma,T)\prt*{\gamma_2-\gamma_1}\prt*{1+\abs*{\ln\delta}}.
\end{equation}
Next, we focus on the term $B$. Since $\gamma_2-\gamma_1 < 1$ and $\gamma_1>1$, $\gamma_2 < 2\gamma_1$. Then,
\begin{equation*}
    u\1^{\gamma_1\prt*{\frac{\gamma_2}{\gamma_1}-2}} \leq C(\Gamma,T)\delta^{-\gamma_1\prt*{2-\frac{\gamma_2}{\gamma_1}}},
\end{equation*}
and we have
\begin{equation}
\label{eq:B_estimate}
    |B| \leq C(\Gamma,T)\prt*{\gamma_2-\gamma_1}\delta^{-\gamma_1\prt*{2-\frac{\gamma_2}{\gamma_1}}}.
\end{equation}
Putting together equation~\eqref{eq:L_estimate} and estimates~\eqref{eq:A_estimate} and~\eqref{eq:B_estimate}, we deduce
\begin{align*}
    &\norm*{\Delta u\1^{\gamma_1}-\Delta u\1^{\gamma_2}}_{L^1((0,T)\times\T^d)}\\
    &\quad\leq \norm*{A}_{L^\infty((0,T)\times\T^d)}\norm*{\Delta u\1^{\gamma_1}}_{L^1((0,T)\times\T^d)} + \norm*{B}_{L^\infty((0,T)\times\T^d)}\norm*{\nabla u\1^{\gamma_1}}_{L^2((0,T)\times\T^d)}\\
    &\quad\leq C(\Gamma, T)\prt*{1+|\ln\delta|+\delta^{-a}}\abs*{\gamma_2-\gamma_1},
\end{align*}
where
\begin{equation*}
    a = 2\min\prt*{\gamma_1,\gamma_2}-\max\prt*{\gamma_1,\gamma_2},
\end{equation*}
as claimed.
\end{proof}

Below we prove the main result of this section, the conclusion of Theorem~\ref{thm:main} in the regularised situation.
\begin{lemma}
    \label{lem:ThmRegular}
        Let $u\1$ and $u\2$ be two solutions of~\eqref{nonlinear_diffusion_regular} with initial data $u_{0,\gamma_1}$ and $u_{0,\gamma_2}$. Then
    \begin{equation*}
        \label{RegEstimate}
    \begin{aligned}    
        \sup_{t\in[0,T]}\norm*{u\1(t) - u\2(t)}_{L^1(\T^d)} &\leq e^{G_M T}\norm*{u_{0,\gamma_1} - u_{0,\gamma_2}}_{L^1(\T^d)}\\
        &\quad+ C(\Gamma,L)\prt*{1+|\ln\delta|+\delta^{-a}}\abs*{\gamma_2-\gamma_1}.
    \end{aligned}
    \end{equation*}
\end{lemma}
\begin{proof}
    Without loss of generality, let us assume that $\gamma_2\geq\gamma_1$. Consider the difference between the equations for the two densities, namely
    \begin{equation}
    \label{eq:DifferenceReg}
        \partial_t \prt*{u\1 - u\2} - \prt*{\Delta u\1^{\gamma_1}-\Delta u\2^{\gamma_2}} = u\1 G(p\1) - u\2 G(p\2).
    \end{equation}
We rewrite the second term as
\begin{equation*}
    \Delta u\1^{\gamma_1}-\Delta u\2^{\gamma_2} = \Delta\prt*{u\1^{\gamma_2} - u\2^{\gamma_2}} + \Delta\prt*{u\1^{\gamma_1} - u\1^{\gamma_2}},
\end{equation*}
and the reaction term as
\begin{equation*}
    u\1 G(p\1) - u\2 G(p\2) = \prt*{u\1 - u\2} G(p\2) + u\1\prt*{G(p\1) - G(p\2)}.
\end{equation*}
Now, mulitplying~\eqref{eq:DifferenceReg} by $\sign(u\1-u\2) = \sign(u\1^{\gamma_2}-u\2^{\gamma_2})$ and using Kato's inequality, we obtain
\begin{equation*}
    \begin{aligned}
        &\partial_t\abs*{u\1 - u\2} - \Delta \abs*{u\1^{\gamma_2} - u\2^{\gamma_2}} - \abs*{u\1 - u\2} G(p\2)\\
        &\leq \Delta\prt*{u\1^{\gamma_1} - u\1^{\gamma_2}}\sign(u\1-u\2) + u\1\prt*{G(p\1) - G(p\2)}\sign(u\1-u\2).
    \end{aligned}
\end{equation*}
Integrating in space leads to
\begin{equation}
\label{eq:DifferenceReg2}
\begin{aligned}
    \ddt \norm*{u\1(t) - u\2(t)}_{L^1(\T^d)} &\leq G_M \norm*{u\1(t) - u\2(t)}_{L^1(\R^d)} + \norm*{\Delta\prt*{u\1^{\gamma_1} - u\1^{\gamma_2}}}_{L^1(\T^d)}\\
    &\quad +\int_{\T^d}u\1\prt*{G(p\1) - G(p\2)}\sign(u\1-u\2) \dx x, 
\end{aligned}
\end{equation}
where $G_M=\max_{0\leq p \leq P_H}|G(p)|$ is a finite constant.
Given the result of Lemma~\ref{lem:LapEstimate}, it remains to estimate the reaction term in~\eqref{eq:DifferenceReg2}.
Applying the Mean Value Theorem, we have
\begin{equation*}
    G(p\1) - G(p\2) = G'(\bar p)\prt*{p\1-p\2},
\end{equation*}
for some intermediate value $\bar p \in (\min(p\1,p\2), \max(p\1,p\2))$. We now rewrite the difference of the pressures as follows:
\begin{align*}
    p\1-p\2 &= \frac{\gamma_1}{\gamma_1-1}u\1^{\gamma_1-1} -  \frac{\gamma_2}{\gamma_2-1}u\2^{\gamma_2-1}\\
    &= \frac{\gamma_2}{\gamma_2-1}\prt*{u\1^{\gamma_1-1}-u\2^{\gamma_2-1}} + \prt*{\frac{\gamma_1}{\gamma_1-1}-\frac{\gamma_2}{\gamma_2-1}}u\1^{\gamma_1-1}\\
    &= \frac{\gamma_2}{\gamma_2-1}\brk*{\prt*{u\1^{\gamma_2-1}-u\2^{\gamma_2-1}} - \prt*{u\1^{\gamma_2-1}-u\1^{\gamma_1-1}}} + \frac{\gamma_2-\gamma_1}{(\gamma_1-1)(\gamma_2-1)}u\1^{\gamma_1-1}.
\end{align*}
Therefore, we can consider the integral of~\eqref{eq:DifferenceReg2} as three different terms, which we investigate separately. First, we observe
\begin{align*}
    \frac{\gamma_2-\gamma_1}{(\gamma_1-1)(\gamma_2-1)}&\int_{\T^d} G'(\bar p) u\1 u\1^{\gamma_1-1} \sign(u\1-u\2) \dx x \\ 
    &\leq \frac{\gamma_2-\gamma_1}{(\gmin-1)^2}\norm{G'}_{L^\infty}\norm{u\1(t)}_{L^{\gamma_1}(\T^d)}^{\gamma_1}\\
    &\leq C(\Gamma, L)\prt*{\gamma_2-\gamma_1}.
\end{align*}
Second, since $G$ is a decreasing function,
\begin{equation*}
    \frac{\gamma_2}{\gamma_2-1} u\1 \prt*{u\1^{\gamma_2-1}-u\2^{\gamma_2-1}}\sign(u\1-u\2) G'(\bar p) = \frac{\gamma_2}{\gamma_2-1} u\1 \abs*{u\1^{\gamma_2-1}-u\2^{\gamma_2-1}} G'(\bar p) \leq 0,
\end{equation*}
hence the corresponding integral can be dropped from~\eqref{eq:DifferenceReg2}. Finally, using the Mean Value Theorem again, we write
\begin{equation*}
    u\1^{\gamma_2-1}-u\1^{\gamma_1-1} = (\gamma_2-\gamma_1) u\1^{\bar\gamma-1}\ln u\1,\quad\text{for some $\bar\gamma\in(\gamma_1,\gamma_2)$}.
\end{equation*}
We can then estimate the last integral as
\begin{align*}
    -\frac{\gamma_2}{\gamma_2-1}&\int_{\T^d}u\1 G'(\bar p)\prt*{u\1^{\gamma_2-1}-u\1^{\gamma_1-1}}\sign(u\1-u\2) \dx x\\
    &\leq \frac{\gmin}{\gmin-1}\norm{G'}_{L^\infty}\prt*{\gamma_2-\gamma_1} \int_{\T^d} |u\1^{\bar\gamma}\ln u\1|\dx x\\
    &\leq \frac{\gmin}{\gmin-1}\norm{G'}_{L^\infty}\prt*{\frac{1}{e\bar\gamma}\abs{\T^d} + \norm{u\1}_{L^{\bar\gamma+1}(\T^d)}^{\bar\gamma+1}}\prt*{\gamma_2-\gamma_1}\\
    &\leq C(\Gamma, L)\prt*{\gamma_2-\gamma_1}.
\end{align*}

Returning to~\eqref{eq:DifferenceReg2}, we now have
\begin{equation*}
\label{eq:DifferenceReg3}
\begin{aligned}
    \ddt \norm*{u\1(t) - u\2(t)}_{L^1(\T^d)} \leq G_M \norm*{u\1(t) - u\2(t)}_{L^1(\T^d)} &+ C(\Gamma,L)\abs*{\gamma_2-\gamma_1}\\
    & +\norm*{\Delta\prt*{u\1^{\gamma_1} - u\1^{\gamma_2}}}_{L^1(\T^d)}.
\end{aligned}
\end{equation*}
Using Lemma~\ref{lem:LapEstimate}, this becomes
\begin{equation*}
\label{eq:DifferenceReg4}
\begin{aligned}
    &\ddt \norm*{u\1(t) - u\2(t)}_{L^1(\T^d)}\\
    &\quad\leq G_M \norm*{u\1(t) - u\2(t)}_{L^1(\T^d)} + C(\Gamma,L)\prt*{1+|\ln\delta|+\delta^{-a}}\abs*{\gamma_2-\gamma_1}.
\end{aligned}
\end{equation*}
Applying Gronwall's lemma, we deduce
\begin{align*}
    &\norm*{u\1(t) - u\2(t)}_{L^1(\T^d)}\\ &\quad\leq e^{G_M t}\norm*{u_{0,\gamma_1} - u_{0,\gamma_2}}_{L^1(\T^d)} + C(\Gamma,L)\prt*{1+|\ln\delta|+\delta^{-a}}\abs*{\gamma_2-\gamma_1},
\end{align*}
as required.
\end{proof}

\section{Proof of the main theorem}
\label{sec:MainProof}
First, let us recall that, without loss of generality, we only consider the case when $0<|\gamma_2-\gamma_1|<1$.
Given two solutions as in the statement of Theorem~\ref{thm:main}, we apply the triangle inequality to write
\begin{equation*}
\label{Triangle}
\begin{aligned}
     &\norm*{u\1(t) - u\2(t)}_{L^1(\R^d)} = \norm*{u\1(t) - u\2(t)}_{L^1(\T^d)}\\
     &\;\leq \norm*{u\1(t) - u_{\gamma_1,\delta}(t)}_{L^1(\T^d)} +  \norm*{u\2(t) - u_{\gamma_2,\delta}(t)}_{L^1(\T^d)} +  \norm*{u_{\gamma_1,\delta}(t) - u_{\gamma_2,\delta}(t)}_{L^1(\T^d)}.
\end{aligned}
\end{equation*}
By virtue of Proposition~\ref{prop:ContractionPrinciple} and Lemma~\ref{lem:ThmRegular}, we have
\begin{align*}
    &\norm*{u\1(t) - u\2(t)}_{L^1(\R^d)}\\
    &\leq C(T)\norm*{u_{0,\gamma_1} - u_{0,\gamma_2}}_{L^1(\R^d)} + C(T)\abs{\T^d}\delta + C(\Gamma,L)\prt*{1+|\ln\delta|+\delta^{-a}}\abs*{\gamma_2-\gamma_1}\\
    &\leq C(T)\norm*{u_{0,\gamma_1} - u_{0,\gamma_2}}_{L^1(\R^d)} + C(\Gamma,L)\abs*{\gamma_2-\gamma_1}\\
    &\quad+ C(\Gamma,L)\brk*{\prt*{|\ln\delta|+\delta^{-a}}\abs*{\gamma_2-\gamma_1}+\delta}.
\end{align*}
This inequality holds for any $\delta>0$, but the right-hand side blows up as $\delta\to 0$ and as $\delta\to\infty$. Therefore, we need to optimise in $\delta$ in terms of $|\gamma_2-\gamma_1|$.
First, we observe the bound
\begin{equation*}
    \abs*{\ln\delta} \leq \delta + \delta^{-a}.
\end{equation*}
Now, we investigate the strictly convex function
\begin{equation*}
    (0,\infty)\ni \delta \mapsto \delta + \prt*{\delta + \delta^{-a}}|\gamma_2-\gamma_1|,
\end{equation*}
which attains its unique global minimum at
\begin{equation*}
    \delta_{\mathrm{min}} = \prt*{\frac{a}{\abs*{\gamma_2-\gamma_1}+1}}^{\frac{1}{a+1}}\abs*{\gamma_2-\gamma_1}^{\frac{1}{a+1}}.
\end{equation*}
With this value of $\delta$, we arrive at the bound
\begin{equation}
    \label{eq:FinalBound_0}
\begin{aligned}
    \norm*{u\1(t) - u\2(t)}_{L^1(\R^d)}
    &\leq C(T)\norm*{u_{0,\gamma_1} - u_{0,\gamma_2}}_{L^1(\R^d)} + C(\Gamma,L)\abs*{\gamma_2-\gamma_1}\\
    &\quad+ C(\Gamma,L)\brk*{\prt*{\delta_{\mathrm{min}}+\delta_{\mathrm{min}}^{-a}}\abs*{\gamma_2-\gamma_1}+\delta_{\mathrm{min}}}.
\end{aligned}
\end{equation}
Notice that
\begin{equation*}
     \prt*{\frac{a}{\abs*{\gamma_2-\gamma_1}+1}}^{\frac{1}{a+1}} \leq a^{\frac{1}{a+1}} \leq 2.
\end{equation*}
Therefore, we have $\delta_{\mathrm{min}} \leq 2\abs*{\gamma_2-\gamma_1}^{\frac{1}{a+1}}$, and consequently
\begin{equation*}
    \delta_{\mathrm{min}}\abs*{\gamma_2-\gamma_1} + \delta_{\mathrm{min}} \leq C\abs*{\gamma_2-\gamma_1}^{\frac{1}{a+1}}.
\end{equation*}
For the remaining term, we observe that
\begin{align*}
    \delta_{\mathrm{min}}^{-a}\abs*{\gamma_2-\gamma_1} &= \prt*{\frac{a\abs*{\gamma_2-\gamma_1}}{\abs*{\gamma_2-\gamma_1}+1}}^{-\frac{a}{a+1}}\abs*{\gamma_2-\gamma_1}\\
    &= \prt*{\frac{1}{a}}^{\frac{a}{a+1}}\prt*{1+\frac{1}{\abs*{\gamma_2-\gamma_1}}}^{\frac{a}{a+1}}\abs*{\gamma_2-\gamma_1}\\
    &\leq C\prt*{\frac{2}{\abs*{\gamma_2-\gamma_1}}}^{\frac{a}{a+1}}\abs*{\gamma_2-\gamma_1}\\
    &\leq C\abs*{\gamma_2-\gamma_1}^{\frac{1}{a+1}}.
\end{align*}
Using the last two bounds, estimate~\eqref{eq:FinalBound_0} becomes
\begin{align*}
    \norm*{u\1(t) - u\2(t)}_{L^1(\R^d)}
    &\leq C(T)\norm*{u_{0,\gamma_1} - u_{0,\gamma_2}}_{L^1(\R^d)} + C(\Gamma,L)\abs*{\gamma_2-\gamma_1}^{\frac{1}{a+1}}\\
    &\leq C(T)\norm*{u_{0,\gamma_1} - u_{0,\gamma_2}}_{L^1(\R^d)} + C(\Gamma,L)\abs*{\gamma_2-\gamma_1}^{\frac{1}{\min(\gamma_1,\gamma_2)+1}}.
\end{align*}
\qed

\begin{remark}
    Let us point out an immediate extension of Theorem~\ref{thm:main}. In a more realistic model, the growth function $G$ should also depend on a parameter, say $G=G(\beta, p)$, where $\beta \in [\beta_{\mathrm{min}}, \beta_{\mathrm{max}}]$. Assuming that the dependence of $G$ on $\beta$ is Lipschitz, we easily deduce stability of model~\eqref{nonlinear_diffusion_0} with respect to the pair of parameters $(\beta, \gamma)$. Indeed, we have
    \begin{align*}
        \norm*{u_{\beta_1, \gamma_1}(t) - u_{\beta_2, \gamma_2}(t)}_{L^1(\R^d)} &\leq \norm*{u_{\beta_1, \gamma_1}(t) - u_{\beta_1, \gamma_2}(t)}_{L^1(\R^d)} + \norm*{u_{\beta_1, \gamma_2}(t) - u_{\beta_2, \gamma_2}(t)}_{L^1(\R^d)}\\
        &\lesssim \abs*{\beta_1-\beta_2} + \abs*{\gamma_2-\gamma_1}^{\frac{1}{\min(\gamma_1,\gamma_2)+1}}.
    \end{align*}
\end{remark}

\section{Conclusions and perspectives}\label{sec:Discussion}

In this paper we addressed an analytical problem of significant practical relevance. We consider a well-known macroscopic tissue growth model wherein the cancer cell density obeys a nonlinear diffusion equation with a growth term. Our main result is to demonstrate the stability of the solutions of this model with respect to the diffusion exponent. This result is necessary to apply the inverse problem methodology, which arises at least in two classes of problems (i) parameter estimation based on experimental data sets and (ii) qualitative inverse problems, the aim of which is to reverse engineer the studied phenomena. The benefits of using the inverse problem approach, in which model parameters are determined based on observed values, are nowadays widely acknowledged particularly in various engineering problems. However, this concept was not sufficiently used in the mathematical modelling of cancer growth and spread, or more generally in mathematical oncology. An example of the application of the inverse problem approach in cancer modelling are the recent works on the nonlocal model of cell proliferation in which the Bayesian inference is used to validate and calibrate the proposed model~\cite{zuzanna_szymanska_2021,gwiazda2023}.

The methodology of inverse problems arises in both deterministic approaches, involving the minimization of a predefined cost function, and stochastic approaches, based on Bayesian methods. 
In the former framework, one seeks to choose those values of the model parameters, so that the solution of the resulting PDE yields the best possible approximation of the measured data. In the context of an in vitro experiment for tumour growth, we can assume that the number of cells is known in disjoint cubes $Q_i$ of side length $1/n$, where $i=1,\dots,n$ and $Q_i$ decompose the total rectangular domain $Q$. Expressing these numbers of cells in terms of a cell number density, $u$, we assume that the following quantities are known
\begin{equation*}
    a_i := \int_{Q_i} u(t,x)\dx x,
\end{equation*}
for $i=1,\dots, n$. The space of all possible densities $u$ can then be described as
\begin{equation*}
    A_n:=\set*{u\in BV(Q),\; \norm*{u}_{BV}\leq 1 \; \big|\; a_i := \int_{Q_i} u(t,x)\dx x\quad \forall i=1,\dots, n}.
\end{equation*}
In this setting, we may seek to solve the minimisation problem
\begin{equation}
\label{L1_min}
    \inf_{\gamma}\; \norm{u_\gamma - u}_{L^1(Q)},
\end{equation}
where $u\in A_n$ and $u_\gamma$ is the density predicted by equation~\eqref{eq:PDEintro}.
A simple calculation shows that the $L^1$-diameter of the set $A_n$ is of order $1/n$. Therefore, the above problem is well-defined in the sense that in the limit $n\to\infty$, the error calculated in~\eqref{L1_min} does not depend on the choice of the representative from $A_n$. Let us however point out the importance of the choice of the functional setting:\ if the set $A_n$ required merely that $u\in L^1$, its $L^1$-diameter would be constant in $n$ and the minimisation of the $L^1$-error would not make much sense. The tissue growth model we consider is well-suited for this situation -- it does propagate $BV$ regularity and we will establish its stability in the $L^1$-norm. Of course, one may consider the minimisation of more general error functionals in~\eqref{L1_min}, based upon more complicated cost functions. This approach is well known for instance in the context of traffic flow modelling~\cite{ColomboCorli}.

Within the Bayesian approach, the model's parameters are treated as random variables described with probability distributions. 
As mentioned earlier, the model may typically depend on a vector of parameters. However, in the context of our current discussion, we deal with a one-dimensional vector comprising only one parameter, $\gamma$.
The first step in the procedure of model calibration is to define $\pi(\gamma)$, the prior probability density, that is the probability density of $\gamma$ without any knowledge of data. Of course, this is subject to choice based on intuition, experience, assumptions, or even a simple guess. Subsequently, one defines a model for data collection described by the likelihood function $L(\textsc{Data}|\gamma)$ that quantifies the probability of observing data, given the parameter $\gamma$.
Bayesian inference is based on posterior distribution $\pi(\gamma|\textsc{Data})$, that is the conditional distribution of a parameter $\gamma$, given the observed data. The posterior distribution is proportional to $\pi(\gamma)L(\textsc{Data}|\gamma)$, according to Bayes' Theorem. The procedure involves updating our belief about a parameter based on experimental data. Typically, obtaining an analytic form of the joint posterior distribution for the basic parameters is unfeasible. Instead, their marginal distributions need to be estimated using an appropriate numerical method.

Without the stability of solutions with respect to parameters, solving the inverse problem becomes unfeasible, as both stochastic and deterministic approaches rely on numerical approximation, which involves addressing finite-dimensional problems. Note that the finite-dimensional approximation is stable with respect to parameter $\gamma$. Thus, if instability arises while solving equation~\eqref{nonlinear_diffusion}, it indicates that, for certain values of $\gamma$, the numerical scheme experiences uncontrolled errors. Moreover, the Bayesian inference needs another numerical procedure to explore posterior $\pi(\gamma|\textsc{Data})$. The stability of the solution concerning $\gamma$ is also required to design an efficient Monte Carlo algorithm. For instance, the performance of the Metropolis-Hastings algorithm, which is a widely used MC method, depends on the acceptance rate of proposed Markov chain states. Typically, the likelihood $L$ is a locally Lipshitz function of solution $u_\gamma$ and further, given our main result, the logarithm of the acceptance ratio from state $\gamma$ to proposed $\gamma'$ can be bounded by
\begin{equation*}
\log\Bigg( \frac{\pi(\gamma'|\textsc{Data})}{\pi(\gamma|\textsc{Data})} \frac{Q(\gamma',\gamma)}{Q(\gamma,\gamma')}\Bigg) \leq C|\gamma - \gamma'|^{\tfrac{1}{1+\min(\gamma,\gamma')}}
\end{equation*} 
for some constant $C$, where $Q$ is the proposal kernel. Thus, a slight perturbation in $\gamma$ would likely be accepted, a scenario that becomes unattainable if the solution proves unstable concerning $\gamma$. This instability presents analogous challenges for other MCMC algorithms.

\subsection*{Acknowledgements} 
\noindent 
T.D.\ acknowledges the support of the National Agency for Academic Exchange, agreement no. BPN/BDE/2023/1/00011/U/00001.
P.G., B.M., and Z.S. are happy to acknowledge the support from the ''Inverse Problems in Tissue Growth Models'' grant funded by the Excellence Initiative - Research University Programme at the University of Warsaw.

\end{document}